\documentclass[12pt,oneside,leqno]{amsart}
\usepackage{amsthm}
\usepackage{amsmath}
\usepackage{amssymb}
\usepackage{amsfonts}
\usepackage{amscd}
\usepackage{amsfonts}
\usepackage{amsthm}
\usepackage{graphicx}
\usepackage[T1]{fontenc}
\usepackage[utf8]{inputenc}
\usepackage[french,english]{babel}
\usepackage{color}

\usepackage{graphicx}
\usepackage{amssymb}
\usepackage{amsfonts}
\usepackage{amsmath}
\usepackage{color}

\usepackage{marvosym}

\issueinfo{}
  {}
  {}
  {}
\pagespan{1}{}
\allowdisplaybreaks
\theoremstyle{plain}
\newtheorem{theorem}{Theorem}[section]
\newtheorem{proposition}[theorem]{Proposition}

\theoremstyle{definition}

\newtheorem{example}[theorem]{Example}
\theoremstyle{remark}
\newtheorem{remark}[theorem]{Remark}

\begin{document}
\title[]{Rational approximations to fractional powers of self-adjoint positive operators}
\thanks{This work was partially supported by GNCS-INdAM, University of Pisa (Grant
PRA$\_$2017$\_$05) and  FRA-University of Trieste. \\
The authors are members of the INdAM research group GNCS.}
\author{Lidia Aceto}
\address{Lidia Aceto\\ Universit\`{a} di Pisa\\
 Dipartimento di Matematica, via F. Buonarroti 1/C - 56127 Pisa\\
Italy}
\email{lidia.aceto@unipi.it}
\author{Paolo Novati}
\address{Paolo Novati \\ Universit\`{a} di Trieste\\
Dipartimento di Matematica e Geoscienze, via Valerio 12/1,  34127 Trieste\\
Italy}
\email{novati@units.it}
 
\maketitle

\begin{abstract}
We investigate the rational approximation of fractional powers of unbounded
positive operators attainable with a specific integral representation of the
operator function. We provide accurate error bounds by exploiting classical
results in approximation theory involving Pad\'{e} approximants. The
analysis improves some existing results and the numerical experiments proves
its accuracy.
\end{abstract}

%%%%%%%%%%%%%%%%%%%%%%
%%%%%%%%%%%%%%%%%%%%%%

\section{Introduction} 

\label{intro}

Let $\mathcal{H}$ be a separable Hilbert space with inner product $%
\left\langle \cdot ,\cdot \right\rangle $ and corresponding norm $\left\Vert
x\right\Vert =\left\langle x,x\right\rangle ^{1/2}$. Let $\mathcal{L}$ be a
self-adjoint positive operator with spectrum $\sigma (\mathcal{L})\subseteq \lbrack
c,+\infty ),$ $c>0$. Moreover, assume that $\mathcal{L}$ has compact
inverse. This paper deals with the numerical approximation of $\mathcal{L}%
^{-\alpha },$ $0<\alpha <1$, that, in this setting, can be defined through
the spectral decomposition, i.e.,
\[
\mathcal{L}^{-\alpha }u=\sum_{s=1}^{\infty }\mu _{s}^{-\alpha } 
\left\langle u,\varphi_{s}\right\rangle \varphi_{s},
\]
where $\{\varphi _{s}\}_{s=1}^{\infty }$ is the orthonormal system of eigenfunctions of $\mathcal{%
L}$ and $\{\mu _{s}\}_{s=1}^{\infty }$ is the corresponding sequence of positive real
eigenvalues (arranged in order of increasing magnitude and counted according
to their multiplicities). Clearly $\mathcal{L}^{-\alpha }$ is a self-adjoint
compact operator on $\mathcal{H}$. Since the function $\lambda ^{-\alpha }$
is continuous in $[c,+\infty )$, we have that (see e.g. \cite[Theorem 1.7.7]%
{Ring})
\[
\left\Vert \mathcal{L}^{-\alpha }\right\Vert =\sup_{\lambda \in \sigma (%
\mathcal{L})}\left\vert \lambda ^{-\alpha }\right\vert =\mu _{1}^{-\alpha }.
\]

An important and widely studied example comes from certain fractional models
involving the symmetric space fractional derivative $(-\Delta )^{\beta /2}$
of order $\beta $ ($1<\beta \leq 2$) \cite{Ka}; in this situation the
fractional power is generally approximated through the approximation of $%
(-\Delta )^{\beta /2-1}$ \cite{Il05}.

A standard approach to approximate $\mathcal{L}^{-\alpha }$ is by means of $%
\mathcal{L}_{N}^{-\alpha }$ where $\mathcal{L}_{N}$ is a finite dimensional
self-adjoint positive operator representing a discretization of $\mathcal{L}.$
Clearly, improving the sharpness of the discretization the typical
situation is that $\lambda _{\min }(\mathcal{L}_{N})\rightarrow c$ and $%
\lambda _{\max }(\mathcal{L}_{N})\rightarrow +\infty $ ($\lambda _{\min }(%
\mathcal{L}_{N})$ and $\lambda _{\max }(\mathcal{L}_{N})$ denoting the
smallest and the largest eigenvalues of $\mathcal{L}_{N}$).

In this framework, in order to compute $\mathcal{L}_{N}^{-\alpha }$
it is quite natural to employ rational forms. For instance, in \cite{HetV,HetP}
some rational approximations are obtained by considering 
the best uniform rational approximation of $\lambda^{1-\alpha}$ and $\lambda^\alpha$ on the interval $[0,1]$.
Beside, other well established techniques are the ones
based on existing integral representations of the Markov function $\lambda
^{-\alpha }$ and then on the use of suitable quadrature rules that finally
lead to rational approximations of the type%
\[
\mathcal{L}_{N}^{-\alpha }\approx \mathcal{R}_{k-1,k}(\mathcal{L}_{N}),\quad
\mathcal{R}_{k-1,k}(\lambda )=\frac{p_{k-1}(\lambda )}{q_{k}(\lambda )},\quad
p_{k-1}\in \Pi _{k-1},\,q_{k}\in \Pi _{k},
\]
(see e.g.  \cite{Bo,FI,No}).  

In this setting, in \cite{AMN,AN,AN2} the rational forms arise from the use
of the Gauss-Jacobi rule for computing the integral representation (see \cite[Eq. (V.4) p. 116]{RB}) 
\begin{equation}
\mathcal{L}^{-\alpha }=\frac{\sin (\alpha \pi )}{(1-\alpha )\pi }%
\int_{0}^{\infty }(\rho ^{1/(1-\alpha )}I+\mathcal{L})^{-1}d\rho ,
\label{mat}
\end{equation}%
after the change of variable%
\begin{equation}
\rho ^{1/(1-\alpha )}=\tau \frac{1-t}{1+t},\qquad \tau >0.  \label{jss}
\end{equation}%
Working in finite dimension, the asymptotically optimal choice of the
parameter $\tau $, yields an error of type
\begin{equation}
\mathcal{O}\left( \exp \left( -4k\sqrt[4]{\lambda _{\min }(\mathcal{L}%
_{N})/\lambda _{\max }(\mathcal{L}_{N})}\right) \right) ,  \label{erb1}
\end{equation}%
where $k$ is the number of points of the quadrature rule, corresponding to a
$\mathcal{R}_{k-1,k}(\lambda )$ rational form. Of course $\sqrt[4]{\lambda
_{\min }(\mathcal{L}_{N})/\lambda _{\max }(\mathcal{L}_{N})}\rightarrow 0$\
improving the quality of the discretization so that (\ref{erb1}) becomes
meaningless whenever $\mathcal{L}_{N}$ represents an arbitrarily sharp
discretization of $\mathcal{L}$.

The basic aim of the present work is to overcome this problem by working in
the infinite dimensional setting. Using the fact that the Gauss-Jacobi
quadrature on Markov functions is related to the Pad\'{e} approximation, we
derive an expression for the truncation error $\lambda ^{-\alpha }-\mathcal{R%
}_{k-1,k}(\lambda ):=\lambda ^{-\alpha }-\tau ^{-\alpha }R_{k-1,k}(\lambda
/\tau )$  (here  $R_{k-1,k}(\lambda
/\tau)$ denotes the $(k-1,k)$-Pad\'{e} approximant of $(\lambda /\tau )^{-\alpha }$),
that leads to an alternative definition of the parameter $\tau $
independent of the discretization and, at the same time, ensuring an
asymptotically optimal rate of convergence. In particular, we are able to
show that the quadrature nodes for (\ref{mat}) can be defined so that the
error for the computation of $\mathcal{L}^{-\alpha }$ decays approximatively
like
\[
\left\Vert \mathcal{L}^{-\alpha }-\tau ^{-\alpha }R_{k-1,k}\left( \frac{%
\mathcal{L}}{\tau }\right) \right\Vert \approx \sin (\alpha \pi ){c^{-\alpha
}\left( \frac{2ke^{1/2}}{\alpha }\right) ^{-4\alpha },}
\]
and therefore sublinearly. Qualitatively, a similar behavior can also been
observed by working with rational Krylov methods to approximate the action
of $\mathcal{L}^{-\alpha }$ (see \cite{Mo}). The sublinearity appears when
considering unbounded spectra. Using the analysis for unbounded operators,
we also show how to improve quantitatively (\ref{erb1}) whenever we assume
to work with $\mathcal{L}_{N}$. The key point consists in taking $\tau $ in (%
\ref{jss}) dependent on $k.$ 

We remark that all the theory here developed can be easily employed to
compute the action of the unbounded operator $\mathcal{L}^{1-\alpha }$ on a
vector $f\in D(\mathcal{L})$ ($D(\mathcal{L})$ is the domain of $\mathcal{L}$%
), that is $\mathcal{L}^{1-\alpha }f$. This may occurs for instance when
solving equations involving the above mentioned fractional Laplacian. In
this situation, after evaluating $g=\mathcal{L}f$, $\mathcal{L}^{1-\alpha }f$
can be computed using our analysis on $\mathcal{L}^{-\alpha }g$.
Nevertheless, the poles of the rational forms here derived can also be used
to compute $\mathcal{L}^{-\alpha }g$ by means of a rational Krylov method.

The paper is organized as follows. In Section \ref{sec2} we recall the basic
features of the Gauss-Jacobi based rational forms for computing (\ref{mat}).
Section \ref{sec3} contains the error analysis and represents the main
contribute of this paper. In Section \ref{sec4} we revisit the error
analysis for the case of bounded spectra. Finally, in Section \ref{sec6} we
present some numerical experiments that validate the theoretical results.

%%%%%%%%%%%%%%%%%%%%%%%%%%%%%%%%%%%%%%%%%
%%%%%%%%%%%%%%%%%%%%%%%%%%%%%%%%%%%%%%%%%

\section{Background on the Gauss-Jacobi approach}

\label{sec2}

Starting from the representation (\ref{mat}), in order to approximate the
fractional Laplacian in \cite{AN} the authors consider the change of
variable (\ref{jss}), that leads to
\begin{equation}
\mathcal{L}^{-\alpha }=\frac{2\sin (\alpha \pi )\tau ^{1-\alpha }}{\pi }%
\int_{-1}^{1}\left( 1-t\right) ^{-\alpha }\left( 1+t\right) ^{\alpha
-2}\left( \tau \frac{1-t}{1+t}I+\mathcal{L}\right) ^{-1}dt.  \label{nint}
\end{equation}%
Using the $k$-point Gauss-Jacobi rule with respect to the weight function $%
\omega (t)=\left( 1-t\right) ^{-\alpha }\left( 1+t\right) ^{\alpha -1}$ the above
integral is approximated by the rational form
\begin{equation}
\mathcal{L}^{-\alpha }\approx \sum_{j=1}^{k}\gamma _{j}(\eta _{j}I+\mathcal{L%
})^{-1}:=\tau ^{-\alpha }R_{k-1,k}\left( \frac{\mathcal{L}}{\tau }\right) ,
\label{rapp}
\end{equation}%
where the coefficients $\gamma _{j}$ and $\eta _{j}$ are given by%
\begin{equation}
\gamma _{j}=\frac{2\sin (\alpha \pi )\tau ^{1-\alpha }}{\pi }\frac{w_{j}}{%
1+\vartheta _{j}},\qquad \eta _{j}=\frac{\tau (1-\vartheta _{j})}{%
1+\vartheta _{j}};  \label{gamma_eta}
\end{equation}%
here $w_{j}$ and $\vartheta _{j}$ are, respectively, the weights and nodes
of the Gauss-Jacobi quadrature rule.

The choice of $\tau $ in (\ref{jss}) is crucial for the quality of the
approximation attainable by (\ref{rapp}). As already mentioned in the
Introduction, working with bounded operators, it has been shown in \cite{AN}
that asymptotically, that is for $k\rightarrow +\infty $, the optimal choice
is given by%
\begin{equation}
\widetilde{\tau }=\sqrt{\lambda _{\min }(\mathcal{L}_{N})\lambda _{\max }(%
\mathcal{L}_{N})}.  \label{tauopt}
\end{equation}
With this choice and denoting by $\kappa (\mathcal{L}_{N})$ the spectral
condition number of $\mathcal{L}_{N}$ and by $\|
\mathcal{\cdot }\|_{2}$ the induced Euclidean norm, we obtain
\begin{equation}
\left\Vert \mathcal{L}_{N}{}^{-\alpha }-\widetilde{\tau }^{-\alpha
}R_{k-1,k}\left( \frac{\mathcal{L}_{N}}{\widetilde{\tau }}\right)
\right\Vert _{2}\leq C\left( \frac{\sqrt[4]{\kappa (\mathcal{L}_{N})}-1}{%
\sqrt[4]{\kappa (\mathcal{L}_{N})}+1}\right) ^{2k},  \label{erb2}
\end{equation}%
with $C$ independent of $k,$ which is a sharper version of (\ref{erb1}). We remark that $\widetilde{\tau }$ is independent of $k.$ In what follows we shall
follow a different strategy allowing a dependence on $k$ (in any case the
coefficients $\gamma _{j}$ and $\eta _{j}$ completely change with $k$) but
at the same time a `mesh-independence', since we work with the unbounded
operator $\mathcal{L}.$

%%%%%%%%%%%%%%%%%%%%%%%%%%%%%%%%%%%%%%%%%%%%%%%%%
%%%%%%%%%%%%%%%%%%%%%%%%%%%%%%%%%%%%%%%%%%%%%%%%%

\section{Error analysis} 

\label{sec3}

Working with the ratio $\lambda /\tau ,$ where $\lambda \in \lbrack
c,+\infty )$ and $\tau >0,$ as shown in \cite[Lemma 4.4]{fgs} the $k$-point
Gauss-Jacobi quadrature given by (\ref{rapp})-(\ref{gamma_eta}) is such that
$R_{k-1,k}\left( {\lambda }/{\tau }\right) $ corresponds to the $(k-1,k)$-Pad\'{e} approximant of $(\lambda /\tau )^{-\alpha }$ centered at $1$.

In this sense, defining
\begin{equation}
z=1-\frac{\lambda }{\tau },  \label{zeta}
\end{equation}%
in what follows we focus the attention on the $(k-1,k)$-Pad\'{e}
approximation
\[
(1-z)^{-\alpha }\approx R_{k-1,k}(1-z).
\]%
Indicating the truncation error by%
\begin{equation}
E_{k-1,k}(1-z):=(1-z)^{-\alpha }-R_{k-1,k}(1-z),  \label{trerr}
\end{equation}%
we have the following result.
\begin{theorem}
\label{MT}For each integer $k\geq 1,$ the exact representation of the truncation error defined  in (\ref{trerr}) is given by
\begin{equation}
E_{k-1,k}(1-z)=\frac{\Gamma (k+1-\alpha )\Gamma (k+1)}{\Gamma (1-\alpha
)\Gamma (2k+1)}\frac{_{2}F_{1}(k+1,k+\alpha ;2k+1;z)}{_{2}F_{1}(-k,k;\alpha
;z^{-1})}\left( -z\right) ^{k},  \label{trunerr}
\end{equation}
in which $\Gamma$ denotes  the gamma function and $_{2}F_{1}$
the hypergeometric function.
\end{theorem}

\begin{proof}
Since \cite[Eq. (5.2)]{Ba}
\[
(1-z)^{-\alpha }={_{2}F_{1}}(1,\alpha ;1;z),
\]%
the expression for the truncation error is obtained following the analysis given in \cite[Sect. 3]{E}. 
\end{proof} 

\begin{proposition}
For $z<1$, let $v=1-2z^{-1}$ and $\xi $ be defined by
\begin{equation}
v\pm \left( v^{2}-1\right) ^{1/2}=e^{\pm \xi }.  \label{vcs}
\end{equation}%
Then, for large values of $k$ we have
\begin{equation}
E_{k-1,k}(1-z)=4\sin (\alpha \pi )\frac{v-1}{e^{(2k+1)\xi }}\frac{\left(
1+e^{-\xi }\right) ^{-2\alpha }}{\left( 1-e^{-\xi }\right) ^{2(1-\alpha )}}%
\left( 1+\mathcal{O}\left( \frac{1}{k}\right) \right) .  \label{res}
\end{equation}
\end{proposition}

\begin{proof}
Since $z=2/(1-v),$ using \cite[Eqs. (16) and (17) p. 77]{Erd} we have that%
\begin{eqnarray*}
&&_{2}F_{1}(k+1,k+\alpha ;2k+1;z)=4\frac{\Gamma (2k+1)\Gamma (1/2)}{\Gamma
(k+\alpha )\Gamma (k+1-\alpha )}\frac{k^{-1/2}}{(-z)^{k+1}} \\
&&\times e^{-(k+1)\xi }(1-e^{-\xi })^{-3/2+\alpha }(1+e^{-\xi
})^{-1/2-\alpha }\left( 1+\mathcal{O}\left( 1/k\right) \right) ,
\end{eqnarray*}%
\begin{eqnarray*}
&&_{2}F_{1}(-k,k;\alpha ;z^{-1})=\frac{\Gamma (k+1)\Gamma (\alpha )}{2\Gamma
(1/2)\Gamma (k+\alpha )}k^{-1/2} \\
&&\times (1-e^{-\xi })^{1/2-\alpha }(1+e^{-\xi })^{\alpha -1/2}\left(
e^{k\xi }+e^{\pm i\pi \alpha ^{-1/2}}e^{-k\xi }\right) \left( 1+\mathcal{O}%
\left( 1/k\right) \right).
\end{eqnarray*}

Plugging these relations in (\ref{trunerr}) and using the identities $\Gamma
(1/2)=\sqrt{\pi }$ and $\Gamma (\alpha )\Gamma (1-\alpha )=\pi /\sin (\pi
\alpha ),$ we find the result.  
\end{proof}

\begin{proposition}
For large values of $k,$ the following representation for the truncation
error holds%
\begin{equation}
E_{k-1,k}\left( \frac{\lambda }{\tau }\right) =2\sin (\alpha \pi)\left(
\frac{\lambda }{\tau }\right) ^{-\alpha }\left[ \frac{\lambda ^{1/2}-\tau
^{1/2}}{\lambda ^{1/2}+\tau ^{1/2}}\right] ^{2k}\left( 1+\mathcal{O}\left(
\frac{1}{k}\right) \right) .  \label{res2}
\end{equation}
\end{proposition}

\begin{proof}
Using (\ref{vcs}), after some algebra we obtain%
\begin{equation}
2\frac{v-1}{e^{(2k+1)\xi }}\frac{\left( 1+e^{-\xi }\right) ^{-2\alpha }}{%
\left( 1-e^{-\xi }\right) ^{2(1-\alpha )}}=\left( \frac{v+1}{v-1}\right)
^{-\alpha }\frac{1}{\left[ v+\left( v^{2}-1\right) ^{1/2}\right] ^{2k}}.
\label{tmp}
\end{equation}%
Since $z={2}/{(1-v)}$ and $z=1-{\lambda }/{\tau }$ we find%
\[
v=\frac{\lambda +\tau }{\lambda -\tau }.
\]%
Substituting this expression in (\ref{tmp}) and then the result in (\ref{res}%
), we easily obtain the statement. 
\end{proof}

By (\ref{rapp}), (\ref{zeta}) and (\ref{trerr}) we have%
\begin{equation}
\left\Vert \mathcal{L}^{-\alpha }-\tau ^{-\alpha }R_{k-1,k}\left( \frac{%
\mathcal{L}}{\tau }\right) \right\Vert \leq \max_{\lambda \geq c}\,\tau
^{-\alpha }\left\vert E_{k-1,k}\left( \frac{\lambda }{\tau }\right)
\right\vert .  \label{errore}
\end{equation}%
As consequence, a suitable value for $\tau $ can be found by working with (%
\ref{res2}). To this purpose, let us consider the function%
\begin{equation}
f(\lambda ,\tau ):=\left( \frac{\lambda }{\tau }\right) ^{-\alpha }\left[
\frac{\lambda ^{1/2}-\tau ^{1/2}}{\lambda ^{1/2}+\tau ^{1/2}}\right] ^{2k},
\label{fl}
\end{equation}%
which is the $\tau $-dependent factor of (\ref{res2}). We want to solve
\begin{equation}
\min_{\tau >0}\max_{\lambda \geq c}\tau ^{-\alpha }f(\lambda ,\tau ).
\label{minmax}
\end{equation}%
For any fixed $\tau >0$, $f(\lambda ,\tau )\rightarrow +\infty $ for $%
\lambda \rightarrow 0^{+}$, $f(\lambda ,\tau )\rightarrow 0$ for $\lambda
\rightarrow +\infty $, $f(\lambda ,\tau )=0$ for $\lambda =\tau $ (the
minimum) and, by solving $\frac{\partial f(\lambda ,\tau )}{\partial \lambda
}=0$, we find a maximum at
\begin{equation}
\overline{\lambda }=\frac{\left( k+\sqrt{k^{2}+1}\right) ^{2}}{\alpha ^{2}}%
\tau =s_{k}^{2}\frac{4k^{2}}{\alpha ^{2}}\tau ,  \label{lams}
\end{equation}%
where%
\begin{equation}
1<s_{k}^{2}=1+\frac{1}{2k^{2}}+\mathcal{O}(1/k^{4}).  \label{sk}
\end{equation}%
Clearly $\overline{\lambda }>\tau $ and hence%
\[
\max_{\lambda \geq c}\tau ^{-\alpha }f(\lambda ,\tau )=\max \left\{ \tau
^{-\alpha }f(c,\tau ),\tau ^{-\alpha }f(\overline{\lambda },\tau )\right\} .
\]%
Setting%
\begin{equation}
\varphi _{1}\left( \tau \right) :=\tau ^{-\alpha }f(c,\tau ),\quad \varphi
_{2}\left( \tau \right) :=\tau ^{-\alpha }f(\overline{\lambda },\tau ),
\label{ph}
\end{equation}%
by (\ref{fl}) we find%
\begin{eqnarray}
\varphi _{1}\left( \tau \right) &=&c^{-\alpha }\left[ \frac{c^{1/2}-\tau^{1/2}}{c^{1/2}+\tau ^{1/2}}\right] ^{2k},  \nonumber \\
\varphi _{2}\left( \tau \right) &=&\tau ^{-\alpha }f\left( s_{k}^{2}\frac{4k^{2}}{\alpha ^{2}}\tau ,\tau \right)  \nonumber \\
&=&\tau^{-\alpha }\left( s_{k}^{2}\frac{4k^{2}}{\alpha ^{2}}\right)^{-\alpha }\left( \frac{2ks_{k}-\alpha }{2ks_{k}+\alpha }\right) ^{2k}
\nonumber \\
&=&\tau^{-\alpha }\left( \frac{4k^{2}e^{2}}{\alpha ^{2}}\right) ^{-\alpha
}(1+\mathcal{O}(1/k^{2})),  \label{phik}
\end{eqnarray}%
where the last equality follows from (\ref{sk}) and by considering the
Taylor expansion around $y=0$ after setting $s_{k}^{2}=1+y.$ Since $\varphi
_{2}\left( \tau \right) $ is monotone decreasing, whereas $\varphi
_{1}\left( \tau \right) $ is monotone increasing for $\tau >c$, the solution
of (\ref{minmax}) is obtained by solving%
\begin{equation}
\varphi_{1}\left( \tau \right) =\varphi_{2} \left( \tau \right) \mbox{ for }\tau >c.  \label{pbt}
\end{equation}

\begin{proposition}
Let $\tau ^{\ast }$ be the solution of (\ref{pbt}). Then%
\begin{equation}
\tau ^{\ast }\approx \tau _{k}:=c\left( \frac{\alpha }{2ke}\right) ^{2}\exp
\left( 2W\left( \frac{4k^{2}e}{\alpha ^{2}}\right) \right) ,  \label{tauk}
\end{equation}%
where $W$ denotes the Lambert-W function.
\end{proposition}

\begin{proof}
Neglecting the factor $(1+\mathcal{O}(1/k^{2}))$ in (\ref{phik}), equation (%
\ref{pbt}) implies%
\begin{equation}
\left( \frac{c}{\tau }\right) ^{-\alpha }\left[ \frac{\tau ^{1/2}-c^{1/2}}{%
\tau ^{1/2}+c^{1/2}}\right] ^{2k}=\left( \frac{4k^{2}e^{2}}{\alpha ^{2}}%
\right) ^{-\alpha }.  \label{pr}
\end{equation}%
Setting $x=\left( {c}/{\tau }\right) ^{1/2}<1$ and%
\begin{equation}
a_{k}=\frac{\alpha }{2ke},  \label{ak}
\end{equation}%
by (\ref{pr}) we obtain%
\[
x^{-\frac{\alpha }{k}}\left( \frac{1-x}{1+x}\right) =a_{k}^{\frac{\alpha }{k}%
}.
\]%
Using the approximation%
\begin{equation}
\frac{1-x}{1+x}\approx e^{-2x},  \label{apprexp}
\end{equation}%
we solve%
\[
e^{-2x}=(a_{k}x)^{\frac{\alpha }{k}}.
\]%
Therefore
\[
-2x=\frac{\alpha }{k}\ln \left( a_{k}x\right)
\]%
which implies
\[
\frac{2k}{a_{k}\alpha }=\frac{1}{a_{k}x}\ln \left( \frac{1}{a_{k}x}\right) .
\]%
Using the Lambert-W function, the solution for such equation is given by
\[
\frac{1}{a_{k}x}=\exp \left( W\left( \frac{2k}{a_{k}\alpha }\right) \right) .
\]%
Substituting $x$ by $\left( {c}/{\tau }\right) ^{1/2}$ and using (\ref{ak})
we obtain the expression of $\tau _{k}.$ 
\end{proof}

We remark that since for large $z$
\[
W\left( z\right) =\ln z-\ln \left( \ln z\right) +\mathcal{O}(1),
\]%
we have (see (\ref{tauk}))
\begin{equation}
\tau _{k}=c\frac{4k^{2}}{\alpha ^{2}}\left[ \ln \left( \frac{4k^{2}}{\alpha
^{2}}e\right) \right] ^{-2}(1+\mathcal{O}(1/k^{2})).  \label{tk}
\end{equation}

By (\ref{phik}) we thus obtain%
\begin{eqnarray*}
\varphi _{2}\left( \tau _{k}\right) &=&\tau _{k}^{-\alpha }f(\overline{%
\lambda },\tau _{k}) \\
&=&c^{-\alpha }\left( \frac{2ke^{1/2}}{\alpha }\right) ^{-4\alpha }\left[
2\ln \left( \frac{2k}{\alpha }\right) +1\right] ^{2\alpha }(1+\mathcal{O}%
(1/k^{2})).
\end{eqnarray*}%
The above analysis yields the following result.

\begin{theorem}
\label{MT1}Let $\tau _{k}$ be defined according to (\ref{tauk}). Taking $\tau =\tau _{k}$ in (%
\ref{jss}) we have%
\begin{eqnarray}
\left\Vert \mathcal{L}^{-\alpha }-\tau _{k}^{-\alpha }R_{k-1,k}\left( 
\frac{\mathcal{L}}{\tau _{k}}\right) \right\Vert &\leq &2\sin (\alpha \pi)\,c^{-\alpha }\left( 
\frac{2ke^{1/2}}{\alpha }\right) ^{-4\alpha }  \nonumber
\\
&&\times \left[ 2\ln \left( \frac{2k}{\alpha }\right) +1\right] ^{2\alpha
}(1+\mathcal{O}(1/k^{2})).  \label{th1}
\end{eqnarray}
\end{theorem}

\begin{proof}
The statement immediately follows from (\ref{res2}), (\ref{errore}), and the
analysis just made. 
\end{proof}

\begin{remark}
The factor $c^{-\alpha }$ in the bound (\ref{th1}) reveals how the problem
becomes increasingly difficult if the spectrum is close to the branch point
of $\lambda^{-\alpha }.$ 
\end{remark}

%%%%%%%%%%%%%%%%%%%%%%%%%%%%%%%%%%%%%%%%%%%%%%%%%
%%%%%%%%%%%%%%%%%%%%%%%%%%%%%%%%%%%%%%%%%%%%%%%%%

\section{The case of bounded operators}

\label{sec4}

The theory just developed can be easily adapted to the case of bounded
operators $\mathcal{L}_{N}$ with spectrum contained in $[c,\lambda _{N}]$,
where $\lambda _{N}=\lambda _{\max }(\mathcal{L}_{N})$. In this situation we
want to solve%
\begin{equation}
\min_{\tau >0}\max_{c\leq \lambda \leq \lambda _{N}}\tau ^{-\alpha
}f(\lambda ,\tau ).  \label{minmax2}
\end{equation}%
Looking at (\ref{lams}) we have $\overline{\lambda }=\overline{\lambda }%
(k)\rightarrow +\infty $ as $k\rightarrow +\infty $. As a consequence for, $%
\overline{\lambda }\leq \lambda _{N}$ ($k$ small) the solution of (\ref%
{minmax2}) remains the one approximated by (\ref{tauk}) and the bound (\ref%
{th1}) is still valid. On the contrary, for $\overline{\lambda }>\lambda
_{N} $ ($k$ large), the bound can be improved as follows.

Remembering the features of the function $f(\lambda ,\tau )$ introduced in (%
\ref{fl}), we have that for $\overline{\lambda }>\lambda _{N}$ the solution
of (\ref{minmax2}) is obtained by solving%
\begin{equation}
\varphi_{1}\left( \tau \right) =\varphi_{3}\left( \tau \right) \mbox{ for } \tau >c,  \label{pbt2}
\end{equation}%
where $\varphi_{1}\left( \tau \right)$ is defined in (\ref{ph}) and
\[
\varphi_{3}\left( \tau \right) :=\tau ^{-\alpha }f(\lambda_{N},\tau )
=\lambda_{N}^{-\alpha }\left[ \frac{\lambda_{N}^{1/2}-\tau ^{1/2}}{\lambda_{N}^{1/2}+\tau ^{1/2}}\right] ^{2k}.
\]
It can be easily verified that the equation $\varphi _{1}\left( \tau \right)
=\varphi _{3}\left( \tau \right) $ has in fact two solutions, one in the
interval $(0,c)$ and the other in $(c,\lambda _{N})$. Anyway since $\varphi
_{3}\left( \tau \right) $ is monotone decreasing in $[0,\lambda _{N})$ we
have to look for the one in $(c,\lambda _{N})$ as stated in (\ref{pbt2}).

\begin{proposition}
Let ${\hat \tau}^{\ast }$ be the solution of (\ref{pbt2}). Then%
\begin{equation}
{\hat \tau}^{\ast } \approx \tau _{k}:={\left( -\frac{\alpha \lambda
_{N}^{1/2}}{8k} \ln \left( \frac{\lambda _{N}}{c}\right) +\sqrt{\left(\frac{%
\alpha \lambda _{N}^{1/2}}{8k} \ln \left( \frac{\lambda _{N}}{c}%
\right)\right)^2+\left( c\,\lambda _{N}\right) ^{1/2}}\right) ^{2}}.
\label{tk2}
\end{equation}
\end{proposition}

\begin{proof}
From (\ref{pbt2}) we have
\begin{equation}
c^{-\alpha }\left[ \frac{\tau ^{1/2}-c^{1/2}}{\tau ^{1/2}+c^{1/2}}\right]
^{2k}=\lambda _{N}^{-\alpha }\left[ \frac{\lambda _{N}^{1/2}-\tau ^{1/2}}{%
\lambda _{N}^{1/2}+\tau ^{1/2}}\right] ^{2k}.  \label{pr1}
\end{equation}%
Setting $x=\left( {c}/{\tau }\right) ^{1/2}<1$ and $y=\left( \tau /\lambda
_{N}\right) ^{1/2}<1$ by (\ref{pr1}) we obtain
\[
\left( \frac{1-x}{1+x}\right) =\left( \frac{\lambda _{N}}{c}\right) ^{-\frac{%
\alpha }{2k}}\left( \frac{1-y}{1+y}\right) .
\]%
Using (\ref{apprexp}) we solve%
\[
e^{-2x}=\left( \frac{\lambda _{N}}{c}\right) ^{-\frac{\alpha }{2k}}e^{-2y}.
\]%
Therefore
\[
-2x=-\frac{\alpha }{2k}\ln \left( \frac{\lambda _{N}}{c}\right) -2y
\]%
which implies
\[
x-y=\frac{\alpha }{4k}\ln \left( \frac{\lambda _{N}}{c}\right) .
\]%
Substituting $x$ by $\left( {c}/{\tau }\right) ^{1/2}$ and $y$ by $\left(
\tau /\lambda _{N}\right) ^{1/2}$ after some algebra we obtain
\[
\tau +\frac{\alpha }{4k}\lambda _{N}^{1/2}\ln \left( \frac{\lambda _{N}}{c}%
\right) \tau ^{1/2}-(c\,\lambda _{N})^{1/2}=0.
\]%
Then, solving this equation and taking the positive solution, we obtain the
expression of $\tau _{k}.$   
\end{proof}

Observe that by (\ref{tk2}), for $k\rightarrow +\infty $ we have
\begin{eqnarray*}
\left( \frac{\tau _{k}}{\lambda _{N}}\right) ^{1/2} &=&-\frac{\alpha }{8k}%
\ln \left( \frac{\lambda _{N}}{c}\right) +\sqrt{\left( \frac{\alpha }{8k}\ln
\left( \frac{{\lambda _{N}}}{c}\right) \right) ^{2}+\left( \frac{c}{\lambda_{N}}\right) ^{1/2}} \\
&=&-\frac{\alpha }{8k}\ln \left(
\frac{{\lambda _{N}}}{c}\right) +\left( \frac{c}{\lambda _{N}}\right) ^{1/4}+\mathcal{O}\left( \frac{1}{k^{2}}\right).
\end{eqnarray*}%
Finally, using (\ref{apprexp}) and the above expression we obtain%
\begin{eqnarray}
\varphi_{3}\left( \tau_{k}\right) &=&\lambda_{N}^{-\alpha }\left[ \frac{%
\lambda_{N}^{1/2}-\tau_{k}^{1/2}}{\lambda_{N}^{1/2}+\tau_{k}^{1/2}}%
\right] ^{2k}  \nonumber \\
&\leq &\lambda _{N}^{-\alpha }\exp \left( -4k\left( \frac{\tau_{k}}{\lambda_{N}}\right) ^{1/2}\right)  \nonumber \\
&=&\lambda_{N}^{-\alpha }\exp \left( -4k\left[ \left( \frac{c}{\lambda_{N}}%
\right) ^{1/4}-\frac{\alpha }{8k}\ln \left( \frac{{\lambda_{N}}}{c}\right) %
\right] \right) \left( 1+\mathcal{O}\left( \frac{1}{k}\right) \right)  \nonumber
\\
&=&\lambda _{N}^{-\alpha }\exp \left( -4k\left( \frac{c}{\lambda _{N}}%
\right) ^{1/4}\right) \exp \left( \frac{\alpha }{2}\ln \left( \frac{{\lambda
_{N}}}{c}\right) \right) \left( 1+\mathcal{O}\left( \frac{1}{k}\right)
\right)  \nonumber \\
&=&\left( c\,\lambda _{N}\right) ^{-\alpha /2}\exp \left( -4k\left( \frac{c}{%
\lambda _{N}}\right) ^{1/4}\right) \left( 1+\mathcal{O}\left( \frac{1}{k}%
\right) \right) .  \label{res3}
\end{eqnarray}%
The above analysis yields the following result.

\begin{theorem}
\label{MT2}Let $\overline{k}$ be such that for each $k\geq \overline{k}$ we
have $\overline{\lambda }=\overline{\lambda }(k)>\lambda _{N}.$ Then for
each $k\geq \overline{k}$, taking in (\ref{jss}) $\tau =\tau _{k},$ where $%
\tau _{k}$ is given in (\ref{tk2}), the following bound holds%
\begin{eqnarray}
\left\Vert \mathcal{L}_{N}^{-\alpha }-\tau _{k}^{-\alpha }R_{k-1,k}\left(
\frac{\mathcal{L}_{N}}{\tau _{k}}\right) \right\Vert &\leq&  2\sin (\alpha
\pi )\left( c\,\lambda _{N}\right) ^{-\alpha /2}  \nonumber \\
&&\! \times \exp \left( -4k\left( \frac{c}{\lambda _{N}}\right) ^{1/4}\right) \!
\left( 1+{\mathcal{O}}\left({1}/{k}\right) \right).  \label{th2}
\end{eqnarray}
\end{theorem}

It is important to remark that, qualitatively, we have obtained the same
result of \cite{AN} and reported in (\ref{erb1}) following a completely
different approach. Nevertheless the analysis here presented is
quantitatively more accurate since it provides more information about the
constant preceeding the exponential factor.

Observe moreover that the analysis of this section may be particularly
useful when, in practical situation, one is forced to keep the
discretization quite coarse (so that $\overline{k}$ may be rather small) and
also to keep small the number of quadrature nodes $k$. In this case,
defining $\tau _{k}$ as in (\ref{tk2}) may provide results much better than
the one attainable with the asymptotically optimal choice $\tilde \tau =%
\sqrt{\lambda _{\min }(\mathcal{L}_{N})\lambda _{\max }(\mathcal{L}_{N})}$.

%%%%%%%%%%%%%%%%%%%%%%%%%%%%%%%%%%%%%%%%%%%%%%%%%
%%%%%%%%%%%%%%%%%%%%%%%%%%%%%%%%%%%%%%%%%%%%%%%%%

\section{Numerical experiments}

\label{sec6}

In this section we present the numerical results obtained by considering two simple cases of self-adjoint positive operators.
In particular, in the first example we try to simulate the behavior of an unbounded operator by working with a
diagonal matrix with a wide spectrum. In the second one we consider the
standard central difference discretization of the one dimensional Laplace
operator with Dirichlet boundary conditions.

We remark that in all the experiments the weights and nodes of the Gauss-Jacobi quadrature rule are computed  by
using the Matlab function \texttt{jacpts} implemented in Chebfun by Hale and Townsend \cite{HT}.

 \begin{example} \label{e1}
 We define $A=diag(1, 2, \dots,N)$ and $\mathcal{L}_{N}=A^{p}$ so that $\sigma (\mathcal{L}_{N})\subseteq \lbrack 1,N^{p}]$. Taking $N=100$ and $p=4$,
 in Figure \ref{fig1}, for $\alpha=0.25, 0.5, 0.75$  the error (with respect to the Euclidean norm)
 and the error bound (\ref{th1}) are plotted versus  $k,$ the number of points of the used Gauss-Jacobi quadrature rule.

\begin{figure}
\includegraphics[width=1.00\textwidth]{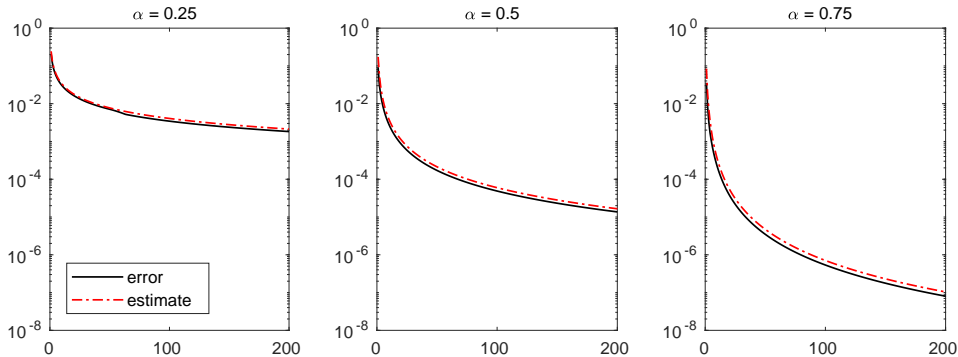}
 \caption{Error and error bound (\ref{th1}) with respect to $k$ for Example 1 with $N=100$, $p=4.$} % and $\alpha=0.25, 0.5, 0.75.$}
 \label{fig1}
 \end{figure}
 
 In Figure \ref{fig2}, for  $\alpha=0.5$ we plot the error obtained using  $\tau _{k}$ taken as in (\ref{tauk}) and  $\tilde \tau$ as in (\ref{tauopt}), changing the amplitude of the
 spectrum, that is, the value of $p.$ In particular, we fix again $N=100$ and take $p=2,3,4.$
 
 \begin{figure}
\includegraphics[width=1.00\textwidth]{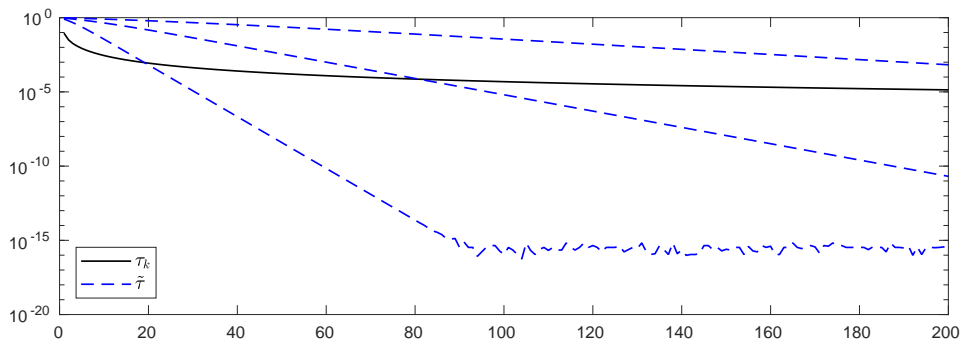}
 \caption{
 Error comparison for Example \ref{e1} using  $\tilde \tau$ as in (\ref{tauopt}) 
and $\tau _{k}$ as in (\ref{tauk}), $p=2,3,4$ (lowest to highest curve),  
$N=100$ and $\alpha =0.5.$}
 \label{fig2}
 \end{figure}
 
 The figure clearly shows the improvement attainable with $\tau _{k}$ for $k$
 small, and moreover the deterioration of the method for very large spectra
 when using $\tilde \tau.$
 \end{example}

\begin{example} \label{e2}
We consider the linear operator $\mathcal{L}u=-u^{\prime \prime},$ $u:[0, b]\rightarrow \mathbb{R},$ 
with Dirichlet boundary conditions $u(0)=u(b)=0$. It is known that $\mathcal{L}$ has a point spectrum consisting
entirely of eigenvalues
\[
\mu_{s}=\frac{\pi ^{2}s^{2}}{b^{2}},\qquad \mbox{for }s=1,2,3,\dots.
\]
Using the standard central difference scheme on a uniform grid and setting $b=1$, in this example we work with the operator
\[
\mathcal{L}_{N}:=(N+1)^{2}tridiag(-1,2,-1)\in \mathbb{R}^{N\times N}.
\]
The eigenvalues are
\[
\lambda_{j}=4(N+1)^{2}\sin ^{2}\left( \frac{j\pi }{2(N+1)}\right) ,\qquad
j=1,2,\dots ,N,
\]
so that $\sigma (\mathcal{L}_{N})\subseteq \lbrack \pi^2,4(N+1)^{2}].$  

The aim of this example is to show the improvement that can be obtained by
using the $k$-dependent parameter $\tau _{k}$ as in (\ref{tk2}) with respect
to the asymptotically optimal one $\tilde \tau.$ By choosing $N=500$, so that $%
\lambda _{N}\approx 10^{6}$, and $\alpha =0.5$, in Figure \ref{fig3} the
errors are reported. In Figure \ref{fig4} we also plot the values of $\tau
_{k}$. We remark that for other choice of $\alpha $ the results are
qualitatively identical.

\begin{figure}
\includegraphics[width=1.00\textwidth]{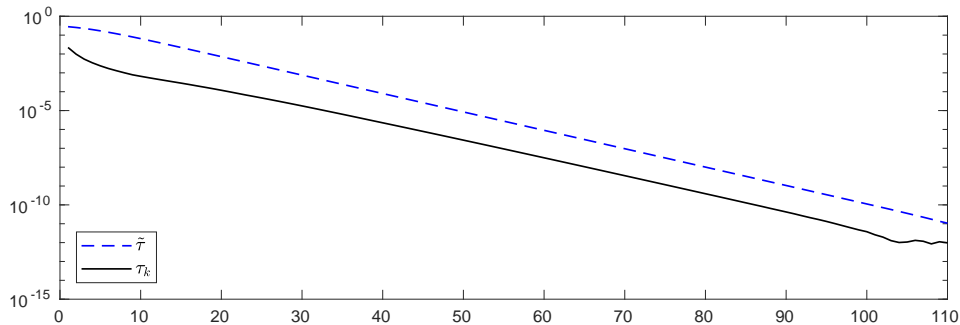}
\caption{Error comparison for Example \ref{e2} using  $\tilde \tau$ as in (\ref{tauopt}) 
and $\tau _{k}$ as in (\ref{tk2}), $N=500$ and  $\alpha =0.5.$}
\label{fig3}
\end{figure}

\begin{figure}
\includegraphics[width=1.00\textwidth]{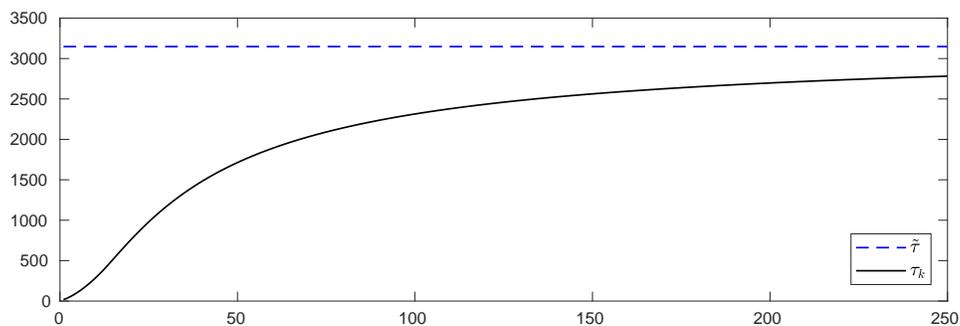}
\caption{Selected values for  $\tau _{k}$ defined by (\ref{tk2})  for Example \ref{e2} with $N=500$ and $\alpha =0.5$.}
\label{fig4}
\end{figure}

Finally, still working with this example, we show the accuracy of the bound (%
\ref{th2}) for $\alpha=0.25, 0.5, 0.75.$  The results are reported in Figure \ref{fig5}.

\begin{figure}
\includegraphics[width=1.00\textwidth]{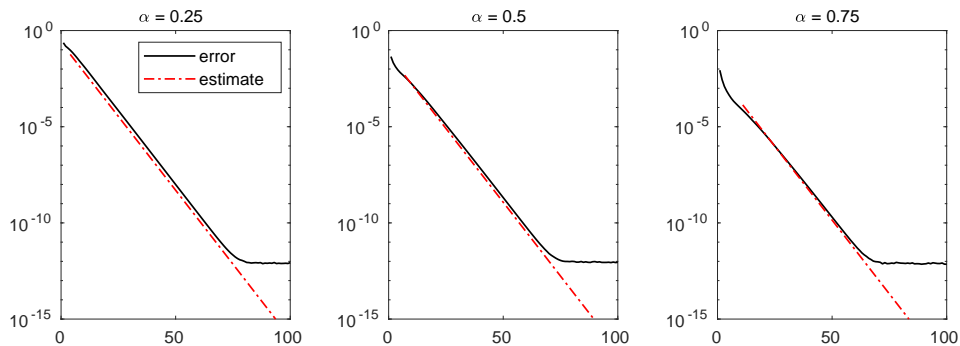}
\caption{Error and error estimate (\ref{th2}) with respect to $k$  for Example 2 with $N=200.$} %, and $\alpha=0.25,0.5,0.75.$}
\label{fig5}
\end{figure}

\end{example}

%%%%%%%%%%%%%%%%%%%%%%%%%%%%%%%%%%%%%%%%%%%%%%%%%
%%%%%%%%%%%%%%%%%%%%%%%%%%%%%%%%%%%%%%%%%%%%%%%%%

\section{Conclusions} 

\label{sec7}

In this paper we have considered rational approximations of fractional powers of unbounded
positive operators obtained by exploiting the connection between Gauss-Jacobi quadrature on
Markov functions and Pad\'{e}  approximants. Using classical results in approximation theory, we
have provided very sharp a priori estimates of the truncation errors that allow to properly
define the parameter $\tau$. The numerical experiments confirm that such analysis improves some existing results.

%%%%%%%%%%%%%%%%%%%%%%%%%%%%%%%%%%%%%%%%%%%%%%%%%
%%%%%%%%%%%%%%%%%%%%%%%%%%%%%%%%%%%%%%%%%%%%%%%%%

%% For one-column wide figures use
%\begin{figure}
%% Use the releSect.vant command to insert your figure file.
%% For example, with the graphicx package use
%  \includegraphics{example.eps}
%% figure caption is below the figure
%\caption{Please write your figure caption here}
%\label{fig:1}       % Give a unique label
%\end{figure}
%%

%\begin{acknowledgements}
%If you'd like to thank anyone, place your comments here
%and remove the percent signs.
%\end{acknowledgements}

% BibTeX users please use one of
%\bibliographystyle{spbasic}      % basic style, author-year citations
%\bibliographystyle{spmpsci}      % mathematics and physical sciences
%\bibliographystyle{spphys}       % APS-like style for physics
%\bibliography{}   % name your BibTeX data base

% Non-BibTeX users please use

\end{document}